\documentclass[11pt]{amsart}
\usepackage{amsmath,amssymb,amsthm}
\usepackage{url,hyperref,setspace,color}
\usepackage{graphicx}
\usepackage{epsfig}
\usepackage{subfigure}
\usepackage{caption}
\usepackage[all]{xy}

\usepackage{geometry,algorithmic,algorithm,verbatim}

\usepackage{fullpage}
\usepackage{cite}

\newtheorem{newthm}{}[section]

\newtheorem{definition}[newthm]{Definition}

\newtheorem{theorem}[newthm]{Theorem}
\newtheorem{lemma}[newthm]{Lemma}
\newtheorem{corollary}[newthm]{Corollary}
\newtheorem{conjecture}[newthm]{Conjecture}
\newtheorem{proposition}[newthm]{Proposition}
\newtheorem{remark}[newthm]{Remark}
\newtheorem{question}[newthm]{Question}

\newcommand{\Sym}{{\mathfrak{S}}}
\newcommand{\Aut}{{\mathrm{Aut}}}

\newcommand\residualedge[1][.5cm]{\rule[0.5ex]{#1}{.6pt}}
\newcommand\staredge{\mbox{%
 \hspace{.5mm}\residualedge[1mm]\hspace{1mm}\residualedge[1mm]\hspace{1mm}\residualedge[1mm]}\hspace{.5mm}}
  
 \newcommand{\verteq}{\rotatebox{90}{$\,=$}}
\newcommand{\equalto}[2]{\underset{\scriptstyle\overset{\mkern4mu\verteq}{#2}}{#1}}

\begin{document}
\title{A complete classification of which $(n,k)$-star graphs are Cayley graphs}

\author{Karimah Sweet}
\address{Department of Mathematics and Statistics \\
 Oakland University \\
 Rochester, MI 48309}
\email{ksweet@oakland.edu}

\author{Li Li}
\address{Department of Mathematics and Statistics \\
 Oakland University \\
 Rochester, MI 48309}
\email{li2345@oakland.edu}
\thanks{The second author was partially supported by the
Oakland University URC Faculty Research Fellowship Award.}

\author{Eddie Cheng}
\address{Department of Mathematics and Statistics \\
 Oakland University \\
 Rochester, MI 48309}
\email{echeng@oakland.edu}

\author{L\'aszl\'o Lipt\'ak }
\address{Department of Mathematics and Statistics \\
 Oakland University \\
 Rochester, MI 48309}
\email{liptak@oakland.edu}

\author{Daniel E. Steffy}
\address{Department of Mathematics and Statistics \\
 Oakland University \\
 Rochester, MI 48309}
\email{steffy@oakland.edu}

\maketitle

\begin{abstract}
The $(n,k)$-star graphs are an important class of interconnection networks that generalize star graphs, which are superior to hypercubes.
In this paper, we continue the work begun by Cheng et al.~(Graphs and Combinatorics 2017) and complete the classification of all the $(n,k)$-star graphs that are Cayley.  

\end{abstract}

\section{Introduction}
The $(n,k)$-star graph, $S_{n,k}$, where $1\leq k<n$, has as its vertices the $k$-permutations on the set $\{1,\dots,n\}$. (A $k$-permutation on $\{1,\dots,n\}$ is an ordered $k$-tuple of distinct elements from the set.) There are two types of edges in $S_{n,k}$. A \textit{star edge} is an edge between two vertices, one of which can be obtained from the other by exchanging the symbols in position 1 and position $i$ for some $2\leq i\leq k$. A \textit{residual edge} is an edge between two vertices that differ only in their first position. Each vertex in $S_{n,k}$ is incident with $k-1$ star edges and $n-k$ residual edges.

The class of $(n,k)$-star graphs was introduced as a potential interconnection structure for distributed processor computer architectures \cite{CC}. 
In this context, such structures are often referred to as interconnection networks, and they are evaluated based on their structural properties.
Hsu and Lin \cite{HL} record recent progress in this area with an extensive bibliography.

There has been much research on the class of $(n,k)$-star graphs studying embeddings, broadcasting, Hamiltonicity and surface area as well as their applicability in theoretical computer science. Recent papers (within the past 3 years) include \cite{LX,CKR,CL,CQS,XLZHG,HHTH,LXF,DCW,WC,CLS,CHYT}. The first major result on Hamiltonicity was given in \cite{HHTH}, which proves that $(n,k)$-star graphs are Hamiltonian; in fact, an $(n,k)$-star graph  remains Hamiltonian if $n-3$ vertices and/or edges are deleted. 
(There is a conjecture that every finite connected Cayley graph contains a Hamiltonian cycle; for related work, see \cite{GM, M2015,M2016}.)
While it is well known that other classes of popular interconnection networks such as the hypercube and star graph are Cayley graphs, it has remained an open question for  $(n,k)$-star graphs. 
Our work settles this question by showing which $(n,k)$-star graphs are Cayley graphs, providing a deeper understanding of the properties of this class of interconnection networks.

Recall that if $G$ is a finite group and $S$ is a set of some of its non-identity elements, the Cayley graph $\Gamma(G,S)$ is the directed graph whose vertex set is $G$, and whose set of arcs contains  an arc from $u$ to $v$ if and only if there is an element $s\in S$ such that $v=us$. If $S\subset G$ is a subset that generates $G$, then $\Gamma(G,S)$ is connected, and if $s\in S$ implies $s^{-1}\in S$, then we can simplify $\Gamma(G,S)$ to be an undirected graph by replacing each pair of opposite arcs with an undirected edge. 

Since $S_{n,1}$ is isomorphic to the complete graph on $n$ vertices, $K_n$, and $S_{n,n-1}$ is isomorphic to the star graph $S_n$, both of which are Cayley for all $n$, we assume throughout the paper that 
$$k\geq 2 \;\textrm{ and } n\geq k+2.$$
In \cite{ CLLSS}, a classification of the $S_{n,k}$ graphs that are Cayley is given in the case that $k=2$, as well as a necessary condition for $S_{n,k}$ to be Cayley for $k=3$. In this paper we use Sabidussi's Theorem \cite[Lemma 4]{Sabidussi} to study for which $n$ and $k$, $S_{n,k}$ is a Cayley graph. Let us first recall Sabidussi's Theorem and give its corollary for $(n,k)$-star graphs.

\begin{theorem}[Sabidussi's Theorem, \cite{Sabidussi}]\label{Sabidussi}
Let $v$ be a vertex of a finite graph $\Gamma$. The following are equivalent:

(i) $\Gamma$ is a Cayley graph; 

(ii) there is a subgroup $G\le \Aut(\Gamma)$ such that the map $\varepsilon: G\to V(\Gamma)$, $g\mapsto g(v)$ is bijective;

(iii) there is a subgroup $G\le \Aut(\Gamma)$ such that $|G|=|V(\Gamma)|$ and the stabilizer group $G_v$ is trivial;

(iv) $\Aut(\Gamma)$ contains a subgroup that acts regularly (i.e. transitively and freely) on $V(\Gamma)$.

\end{theorem}

We need to introduce some notation to state the following corollary, which follows immediately from Sabidussi's Theorem.  Let $P(n,k)=n!/(n-k)!$ be the number of $k$-permutations of $n$.  For a  permutation  $a=(a_1,\dots,a_n)\in \Sym_n$ (in one-line notation), we define
$$\overline{a}:=[a_1,\dots,a_k],$$
which is a $k$-permutation in $n$, hence is a vertex of $S_{n,k}$. 
We say that $a$ is a representative of $\overline{a}$. We denote by $\mathbf{e}$ the identity permutation and thus $\overline{\mathbf{e}}$ is the $k$-permutation $[1,2,\dots,k]$. We will show later (Theorem \ref{thm:Aut is Sn}) that $\Aut(S_{n,k}) \cong \Sym_n\times \Sym_{k-1}$. 

\begin{corollary}\label{Sabidussi corollary}
Assume $k\geq 2$ and $n\geq k+2$.   The following are equivalent:

(i) $S_{n,k}$ is a Cayley graph;

(ii)  there is a subgroup $G\le \Sym_n\times \Sym_{k-1}$ such that the map $\varepsilon: G\to V(S_{n,k})$, $(\mu,\nu)\mapsto \overline{\mu\nu^{-1}}$ is bijective; 

(iii) there is a subgroup $G\le \Sym_n\times \Sym_{k-1}$ such that $|G|=|V(S_{n,k})|=P(n,k)$ and the stabilizer group $G_{\overline{\mathbf{e}}}$ is trivial.
\end{corollary}

There are several advantages of the approach using Sabidussi's Theorem (versus the approach in \cite{CLLSS}):
\medskip

(a) Computationally, to check that $S_{n,k}$ is Cayley, we only need to study the (conjugacy classes of) subgroups of order $P(n,k)$ in $\Sym_n\times\Sym_{k-1}$. This turns out to be a  much more efficient computational approach than our previous approach in \cite{CLLSS} where we constructed groups using generators and relations. In fact, by using this approach we were able to compute many examples of $S_{n,k}$ and eventually come up with the following statement (first as a conjecture, suggested by our computations).

\begin{theorem}\label{main theorem}
For $k\ge 2$, $n\ge k+2$, the graph $S_{n,k}$ is Cayley if and only if either of the following holds:

\begin{itemize}
\item $n=k+2$.

\item $k=2$ and $n$ is a prime power.

\item $k=3$ and $n-1$ is a prime power.

\item $(n,k)$ is one of the following sporadic cases:  $(n,k)=(9,4)$, $(9,6)$, $(11,4)$, $(12,5)$, $(33,4)$,  or  $(33,30)$.
\end{itemize}
\end{theorem}

 (b) If $S_{n,k}$ is a Cayley graph of a group $G$, then Sabidussi's Theorem asserts that we can regard $G$ as a subgroup of $\Sym_n\times\Sym_{k-1}$. The image $H$ of $G$ under the natural projection to $\Sym_n$ turns out to be $k$-homogeneous. Using a classification of $k$-homogeneous groups we can say much more on $H$, hence on $G$. This idea eventually led to the complete proof of the above theorem.
 
\medskip

There is another interesting observation that we would like to point out. Among the sporadic finite simple groups, the first known ones are the Mathieu groups  $M_{11}, M_{12}, M_{22}, M_{23}, M_{24}$. In the paper we show that the $(n,k)$-star graph $S_{11,4}$ (resp. $S_{12,5}$) is a Cayley graph of the Mathieu group $M_{11}$ (resp. $M_{12}$). Moreover, computation shows that all groups of automorphisms acting regularly on the vertices of $S_{11,4}$ (resp.~$S_{12,5}$) are isomorphic to  $M_{11}$ (resp.~$M_{12}$). So it is interesting to ask the following:

\begin{question}
For a given pair $(G,k)$ where $G$ is a finite group and $k$ a positive integer, does there exist a connected $k$-regular graph $S$ such that $G$ is, up to isomorphism, the only group with Cayley graph $S$?
\end{question}

In particular, the answer is affirmative for $(M_{11},10)$ and $(M_{12},11)$. Some special cases are easy (for example, when $|G|=p$ is prime, then $G$ has to be $\mathbb{Z}/p\mathbb{Z}$ and the answer is affirmative if $k$ is even and $2\le k\le p-1$;  when $k=2$, then the answer is affirmative if $G=\mathbb{Z}/n\mathbb{Z}$ with odd $n$),  but in general it seems difficult. A related conjecture is given in \cite[Remarks on Theorem 1.3]{FLWX} which claims that every finite nonabelian simple group $G$ has a GRR of valency 3. This conjecture immediately implies an affirmative answer for $(G,3)$ for every finite nonabelian simple $G$.

In a previous version of the paper, we asked the above question only for finite simple groups $G$ and without restriction on the valence. That question can be easily answered affirmatively using the graphical regular representation (GRR), as pointed out by a referee. Here is the explanation. By definition, a Cayley graph $\Gamma=\Gamma(G,S)$ is called a GRR of $G$ if $Aut(\Gamma)\cong G$. If $\Gamma$ is a GRR of $G$, then the above question has an affirmative answer for $G$. On the other hand, the question of which groups admit GRR was answered completely by Hetzel (1976) and Godsil (1981):  all finite unsolvable group have GRR, and the only finite solvable groups without GRR are abelian groups of exponent greater than 2, generalized dicyclic groups, and 13 exceptional groups \cite[16g]{Biggs}. In particular, all finite simple groups except $\mathbb{Z}/p\mathbb{Z}$ have GRR. Moreover, the case $\mathbb{Z}/p\mathbb{Z}$ is easy as seen above.

\medskip

The paper is organized as follows. In \S2 we determine the automorphism group of $S_{n,k}$, $\Aut(S_{n,k})$. In \S3 we discuss $k$-homogeneous and $k$-transitive groups, in particular we list their classifications. In \S4 and \S5 we prove the main theorem, Theorem \ref{main theorem}.

\medskip

In an earlier attempt at solving this problem, our method in the original manuscript relied heavily on  the classification of 2-transitive and 3/2-transitive groups (which is based on the 15000-page Classification of Finite Simple Groups), wherein we proved that $\Aut(S_{n,k})$ was isomorphic to a semidirect product of permutation groups. It was pointed out to us by a reader (whom we are grateful to) that  by using  $k$-homogeneousness we get a much simpler approach to attack the problem, and that the semidirect product can be replaced by a direct product. These observations have allowed us to only use the Classification of Finite Simple Groups in parts of Lemmas 3.2-3.5, and helped us to complete the project of determining the Cayleyness of all $(n,k)$-star graphs.

\section{The automorphism group $\Aut(S_{n,k})$}
In this section, we determine $\Aut(S_{n,k})$ in order to apply Sabidussi's Theorem. 

Let $\Sym_n$ be the symmetric group on the set $\{1,\dots,n\}$. Let $\Sym_{k-1}\le \Sym_n$ be the subgroup of $\Sym_n$ that only permutes $\{2,3,\dots,k\}$, i.e., the subgroup that fixes $\{1,k+1,\dots,n\}$.

\subsection{The construction of the group homomorphism $\varphi: \Sym_n\times \Sym_{k-1}\to \Aut(S_{n,k})$}
\begin{definition}\label{df:phi}
Define $\varphi: \Sym_n\times \Sym_{k-1}\to \Aut(S_{n,k})$ by $\varphi(\mu,\nu)(\bar{a})=\overline{\mu a\nu^{-1}}$ for any $a\in\Sym_n$, that is, $\varphi(\mu,\nu)\in \Aut(S_{n,k})$ is defined as follows:
$$\quad
\varphi(\mu,\nu): \; [a_1,\dots,a_k]\mapsto [\mu(a_{\nu^{-1}(1)}), \dots, \mu(a_{\nu^{-1}(k)})].
$$

\end{definition}
The lemma below asserts that it is actually a group homomorphism. 

\begin{lemma}\label{lemma:SnXS{k-1} injective}
 The map $\varphi$ is an injective group homomorphism. As a consequence, $S_{n,k}$ is vertex-transitive (which is proved in \cite[Theorem 3]{CC}). 
\end{lemma}
\begin{proof}
(i) Observe that $\varphi(u,v)(\bar{a})$ does not depend on the choice of $a$, that is, $\mu(a_{\nu^{-1}(1)})$, $\dots$, $\mu(a_{\nu^{-1}(k)})$ are determined by $a_1,\dots,a_k$ (thus $a_{k+1},\dots,a_n$ are irrelavant). This is true because $\nu$ maps the set $\{1,\dots,k\}$ to itself.
\smallskip

(ii) For any pair $(\mu,\nu)\in \Sym_n\times\Sym_{k-1}$, we claim that  $\varphi(\mu,\nu)$ is indeed in $\Aut(S_{n,k})$. 

To show that $\varphi(\mu,\nu)$ sends different vertices to different vertices, we assume $\varphi(\mu,\nu)(\bar{a})=\varphi(\mu,\nu)(\bar{b})$ for $a, b\in\Sym_n$. Thus
$\mu(a_{\nu^{-1}(i)})=\mu(b_{\nu^{-1}(i)})$ for $1\le i\le k$. Since the set $\{\nu^{-1}(i) | 1\le i\le k\}=\{1,\dots,k\}$, we have $\mu(a_j)=\mu(b_j)$ for $1\le j\le k$, thus $a_j=b_j$ for $1\le j\le k$ because $\mu$ is an isomorphism. Therefore $\bar{a}=\bar{b}$.

To show that $\varphi(\mu,\nu)$ sends adjacent vertices to adjacent vertices, we assume $\bar{a}$ and $\bar{b}$ are adjacent, and denote $a'=\mu a \nu^{-1}$, $b'=\mu b\nu^{-1}$. If $\bar{a}$ and $\bar{b}$ are joint by a star edge, say $a_1=b_j$, $a_j=b_1$, and $a_i=b_i$ ($i\neq 1, j$), then 
$a'_1=\mu(a_{\nu^{-1}(1)})=\mu(a_{1})=\mu(b_j)=\mu(b_{\nu^{-1}(\nu(j))})=b'_{\nu(j)}$, 
$b'_1=\mu(b_{\nu^{-1}(1)})=\mu(b_{1})=\mu(a_j)=\mu(a_{\nu^{-1}(\nu(j))})=a'_{\nu(j)}$, 
$a'_i=\mu(a_{\nu^{-1}(i)})=\mu(b_{\nu^{-1}(i)})=b'_{i}$ ($i\neq 1,j$). So $\bar{a'}$ and $\bar{a'}$ are joint by a star edge.  A similar argument works for a residual edge. Therefore $\varphi$ is well-defined. 
\smallskip

(iii) Next, we show that $\varphi$ is a group homomorphism. Indeed, for any $(\mu,\nu), (\mu',\nu')\in \Sym_n\times\Sym_{k-1}$, 
 let $\overline{a}$ be any vertex of $S_{n,k}$ and let $a\in \Sym_n$ be any representative of $\overline{a}$. We have
$$
\varphi(\mu,\nu)\varphi(\mu',\nu')(\overline{a})
=\varphi(\mu,\nu)(\overline{\mu'a\nu'^{-1}})
=\overline{\mu\mu'a\nu'^{-1}\nu^{-1}}
=\overline{\mu\mu'a(\nu\nu')^{-1}}
=\varphi(\mu\mu',\nu\nu')(\bar{a}).
$$
Therefore $\varphi(\mu,\nu)\varphi(\mu',\nu')=\varphi((\mu,\nu)\cdot(\mu',\nu'))$. 

\smallskip

(iv) Then, we show that $\varphi$ is injective. Assume that $(\mu,\nu)\in \Sym_n\times \Sym_{k-1}$ satisfies $\varphi(\mu,\nu)={\rm id}_{\Aut(S_{n,k})}$ (the identity automorphism of $S_{n,k}$). That is, for any vertex $[a_1,\dots,a_k]$ of $S_{n,k}$, 
$$[\mu(a_{1}), \dots,\mu(a_{\nu^{-1}(k-1)}),\mu(a_{\nu^{-1}(k)})]=[a_{1},\dots,a_k].$$
(Note that $\nu^{-1}(1)=1$.)
The equality of the last coordinate $\mu(a_{1})=a_{1}$ holds for any $a_1=1,\dots,n$, so $\mu={\rm id}_{\Sym_n}$. Then $a_{\nu^{-1}(i)}=a_i$, thus $\nu^{-1}(i)=i$ for $2\le i\le k$; that is, $\nu={\rm id}_{\Sym_{k-1}}$. Thus $\varphi$ is injective.
\smallskip

(v) Finally, we show the vertex-transitivity of $S_{n,k}$. Indeed, since every vertex of $S_{n,k}$ is of the form $\bar{a}$ for some (non-unique) $a\in\Sym_n$, we can choose $\mu=a$ and $\nu=\mathbf{e}$ (the identity). Then $\varphi(\mu,\nu)(\bar{\mathbf{e}})=\overline{a\mathbf{e}\mathbf{e}}=\overline{a}$. 
\end{proof}

\subsection{Determining $\Aut(S_{n,k})$}
For a vertex $v$ of $S_{n,k}$, we say $u$ is a residual-adjacent (resp. star-adjacent) neighbor of $v$ if $u$ is connected to $v$ by a residual (resp. star) edge. 

\begin{lemma}\label{lem:transposition} 
Assume $2\le a,b,c,d,e,f\le n$ and $a\neq b$, $b\neq c$, $c\neq d$, $d\neq e$, $e\neq f$, $f\neq a$, such that the following equality of permutations holds: 
$$(1,f)(1,e)(1,d)(1,c)(1,b)(1,a)={\rm id}.$$ 
Then $a=c=e$, $b=d=f$.  
\end{lemma}
\begin{proof}
First, observe two simple computations:

-- for three distinct numbers $i,j,l$, the product $(1,i)(1,j)(1,l)=(1,l,j,i)$ has order 4;

-- for two distinct numbers $i,j$, the product $(1,i)(1,j)(1,i)=(i,j)$ has order 2. 

Next, we prove the lemma by cases:

If $a\neq c$ and $d\neq f$, then $a,b,c$ (resp. $d,e,f$) are three distinct numbers, and thus $(1,d,e,f)=(1,f)(1,e)(1,d)=(1,a)(1,b)(1,c)=(1,c,b,a)$, which implies $d=c$ (as well as $e=b$, $f=a$), a contradiction to our assumption.

If $a=c$, then $(1,f)(1,e)(1,d)=(a,b)$ has order 2, so $d=f$, $(d,e)=(1,d)(1,e)(1,d)=(1,f)(1,e)(1,d)=(a,b)$, which implies either ``$a=d$ and $b=e$'' or ``$a=e$ and $b=d$''. The former is impossible since it implies a contradiction $c=d$. So the latter holds, i.e., $a=c=e$, $b=d=f$. 

If $d=f$, then the argument is similar to the $a=c$ case.
\end{proof}

We define an \emph{alternating 6-cycle} to be a 6-cycle with alternative residual and star edges. We define a \emph{star-edge 6-cycle} to be a 6-cycle consisting solely of star edges.

\begin{lemma}\label{lemma:unique 6-cycle}
Let $u$, $v$, $w$ be three vertices in $S_{n,k}$.

(i) If $uv$ is a residual edge and $vw$ is a star edge, then there is a unique alternating 6-cycle containing $uv$ and $vw$.

(ii) If $uv$ and $vw$ are both star edges, then there is a unique star-edge 6-cycle containing $uv$ and $vw$. 
\end{lemma}
\begin{proof}

(i) Denote $v=[a_1,\dots,a_k]$. Assume $u$ is obtained from $v$ by replacing $a_1=i$ by $l$, and $w$ is obtained from $v$ by swapping the first number $a_1$ with the $r$-th number $a_r=j$. Then there is a 6-cycle connecting $v$ and five vertices obtained from $v$ by replacing $(i,j)=(a_1, a_r)$ by $(j,i)$, $(l,i)$, $(i,l)$, $(j,l)$, $(l,j)$, respectively.  

Next, we show that such a 6-cycle is unique. Assume $u\residualedge  v\staredge w\residualedge x\staredge y\residualedge z\staredge u$ is such a cycle (``$\residualedge$'' denotes a residual edge, ``$\staredge$'' denotes a star edge). For simplicity we only prove the special case $v=\overline{\mathbf{e}}=[1,\dots,k]$, $u=[k+1,2,3,\dots,k]$, $w=[2,1,3,\dots,k]$ (the general case is proven in the same way with much more cumbersome notation). Then $x=[p, 1,3,\dots,k]$ for some $k+1\le p\le n$,  $y$ is a permutation of the set $A=\{p, 1,3,\dots,k\}$ (because $x,y$ are star-adjacent), $z$ is a permutation of the set $B=\{k+1,2,\dots,k\}$ (because $u,z$ are star-adjacent). For $yz$ to be a residual edge, the sets $A$ and $B$ must differ by only one number. Therefore $p=k+1$, $x=[k+1,1,3,\dots,k]$, $y=[1,k+1,3,\dots,k]$, $z=[2,k+1,3,\dots,k]$. So the 6-cycle is unique.

(ii) Denote by $s_j$ ($1\le j\le k-1$) the action on $k$-permutations by swapping $a_1$ with $a_{j+1}$. Assume $u=s_j(v)$ and $w=s_k(v)$. Then the 6-cycle 
$$u\staredge \equalto{s_j(u)}{v}\staredge \equalto{s_ks_j(u)}{w}\staredge s_js_ks_j(u)\staredge s_ks_js_ks_j(u)\staredge s_js_ks_js_ks_j(u)\staredge \equalto{(s_ks_j)^3(u)}{u}$$ 
 satisfies the requirement. 
 
Next we check that there is only one such  6-cycle. Equivalently, if $s_as_bs_cs_ds_es_f(u)=u$, then $a=c=e$, $b=d=f$. This follows from Lemma \ref{lem:transposition}.
\end{proof}

For a vertex $x$ in a graph $\Gamma$, we denote by ${\rm Stab}_x$ the subgroup of $\Aut(\Gamma)$ consisting of all automorphisms that fix $x$. We recall the following orbit stabilizer equality \cite[Theorem 1.4A]{DM}: if $\Gamma$ is a vertex-transitive graph with $m$ vertices and $x$ is any vertex, then
\begin{equation}\label{eq:AugGamma}
|\Aut(\Gamma)|=m|{\rm Stab}_{x}|.
\end{equation}

We need the following lemma.

\begin{lemma}\label{lemma: star to star} 
Let $k\ge 2$ and $n\ge k+2$. An edge of $S_{n,k}$ is a residual edge if and only if it is in a 3-cycle.
As a consequence, an automorphism of $S_{n,k}$ sends a residual edge (resp.
~ star edge) to a residual edge (resp.~ star edge).
\end{lemma}
\begin{proof}
A residual edge with an endpoint $[a_1,a_2,\dots,a_n]$ is in a complete graph of $n-k+1$ vertices of the form $\{[i,a_2,\dots,a_k] \,|\,  i\neq a_2,\dots,a_k\}$. Thus all edges in this complete graph are residual, and therefore a residual edge is in a 3-cycle.  In contrast, a star edge $xy$ is not in a 3-cycle. Indeed, without loss of generality we assume the endpoints $x=[1,2,3,\dots,k]$ and $y=[2,1,3,\dots,k]$. Assume the contrary that $xyz$ is a  3-cycle with the third vertex $z=[a_1,\dots,a_k]$. If $xz$ is residual, then $[a_2,\dots,a_k]=[2,\dots,k]$ and $a_1\neq 1,\dots, k$; therefore $yz$ is neither residual because $[a_2,\dots,a_k]\neq [1,3\dots,k]$, nor star because $\{a_1,\dots,a_k\} \neq \{1,\dots,k\}$ as sets. If $xz$ is star, then $yz$ is also star, and $[a_1,\dots,a_k]$ is a permutation of $[1,\dots,k]$ of opposite parity with both $x=[1,2,\dots,k]$ and $y=[2,1,\dots,k]$; but this is impossible, because the permutations $x$ and $y$ are of opposite parity. This shows that a star edge is not in a 3-cycle.
\end{proof}

\begin{theorem}\label{thm:Aut is Sn}
The group homomorphism $\varphi$ in Definition \ref{df:phi} is an isomorphism:
$$\xymatrix{\Sym_n\times \Sym_{k-1}\ar[r]^{\varphi}_{\cong\;} &\Aut(S_{n,k})}$$
Moreover, for a vertex $v$ of $S_{n,k}$, let $u_1,\dots,u_{n-k}$ (resp. $w_1,\dots,w_{k-1}$) be the residual-adjacent (resp. star-adjacent) neighbors of $v$ arranged in any order.  For a vertex $v'$ of $S_{n,k}$, let $u'_1,\dots,u'_{n-k}$ (resp. $w'_1,\dots,w'_{k-1}$) be the residual-adjacent (resp. star-adjacent) neighbors of $v'$ arranged in any order. 
Then there is a unique automorphism $f\in \Aut(S_{n,k})$ sending $v$ to $v'$, $u_i$ to $u'_i$ $(1\le i\le n-k)$, $w_i$ to $w'_i$ $(1\le i\le k-1)$. 
\end{theorem}
\begin{proof}
First, we show that $\varphi$ is an isomorphism. It suffices to show the following inequality (note that  we already have ``$\ge$'' since $\varphi$ is injective by Lemma \ref{lemma:SnXS{k-1} injective}): 
$$|\Aut(S_{n,k})|\le |\Sym_n\times \Sym_{k-1}|=n!(k-1)!.$$
Let $\overline{\mathbf{e}}=[1,\dots,k]\in V(S_{n,k})$. 
Since $S_{n,k}$ is a vertex-transitive graph, \eqref{eq:AugGamma} implies
$$|\Aut(S_{n,k})|=|S_{n,k}|\; |{\rm Stab}_{\overline{\mathbf{e}}}|=\frac{n!}{(n-k)!}|{\rm Stab}_{\overline{e}}|.$$
Thus it suffices to show the following (note that we already have ``$\ge$''):
$$|{\rm Stab}_{\overline{\mathbf{e}}}|\le (n-k)!(k-1)!.$$ 
To show this inequality, note that we have a group homomorphism (which is well-defined because of Lemma \ref{lemma: star to star})
$$\pi:{\rm Stab}_{\overline{\mathbf{e}}}\to\Sym_{n-k}\times \Sym_{k-1}, \quad f\mapsto (f_1,f_2)$$
where $f_1$ is the restriction of $f$  to the set of $n-k$ residual-adjacent neighbors of $\overline{\mathbf{e}}$, and $f_2$ is the restriction of $f$ to the set of  $k-1$ star-adjacent neighbors of $\overline{\mathbf{e}}$.

Since $|\Sym_{n-k}\times \Sym_{k-1}|=(n-k)!(k-1)!$, it suffices to show that $\pi$ is injective, in other words, the following claim:

\smallskip\noindent\emph{Claim}: if $f\in \Aut(S_{n,k})$ is in the kernel of $\pi$, then $f$ is the trivial automorphism, that is, it fixes every vertex. As a consequence, $\pi$ is bijective.
 
\smallskip\noindent\emph{Proof of claim}: let $V$ be the set of vertices $v$ such that $f$ fixes $v$ and all its adjacent vertices. Then $\overline{\mathbf{e}}\in V$ since $f$ is in the kernel of $\pi$.  If $V$ consists of all vertices of $S_{n,k}$ then we are done. Otherwise assume $V$ does not contain all vertices of $S_{n,k}$. Since $S_{n,k}$ is connected, there is a vertex $u\notin V$ that is adjacent to a vertex $v\in V$. We consider in two cases:

Case 1: $uv$ is a residual edge. Then a residual-adjacent neighbor of $u$ is either $v$ or residual-adjacent to $v$, so it is fixed by $f$. So there exists a star-adjacent neighbor $w$ of $u$ not fixed by $f$. By Lemma \ref{lemma:unique 6-cycle}(i), there is a unique alternating 6-cycle containing $uv$ and $wu$, say $$w\staredge u \residualedge v \staredge x \residualedge y\staredge z\residualedge w.$$ Since $f$ fixes $u$, $v$ and $x$, $f$ must fix the  6-cycle (because of the uniqueness), thus $f$ fixes $w$, a contradiction. 
 
Case 2:  $uv$ is a star edge. Let $w$ be adjacent to $u$, we assert that $f$ fixes $w$, thus gives a contradiction. We show this in two cases. If $uw$ is a residual edge, Lemma \ref{lemma:unique 6-cycle}(i) asserts that there is a unique alternating 6-cycle containing $uv$ and $uw$, say  
$w\residualedge u\staredge v \residualedge x \staredge y \residualedge a\staredge w$. Since $f$ fixes $u, v, x$, $f$ must fix the 6-cycle, thus it fixes $w$. If $uw$ is a star edge, Lemma \ref{lemma:unique 6-cycle}(ii) asserts that there is a unique star-edge 6-cycle consisting of $uv$ and $uw$, say 
$w\staredge u\staredge v \staredge x \staredge y \staredge a\staredge w$. 
Since $f$ fixes $u,v,x$, $f$ also fixes $w$. 

This completes the proof of claim. 
\smallskip

Next we show the ``Moreover'' part. Since $S_{n,k}$ is vertex-transitive, there exists $\sigma, \tau\in \Aut(S_{n,k})$ such that $\sigma(v)=\overline{\mathbf{e}}$, $\tau(v')=\overline{\mathbf{e}}$. For any $f$ satisfying the condition,  replacing $f$ by  $\tau f\sigma^{-1}$ if necessary, we can assume that $v=v'=\overline{\mathbf{e}}$. Then the existence and uniqueness of $f$ follows from the above conclusion that $\pi$ is bijective.
\end{proof}

\begin{remark}
The map $\varphi$ is not surjective if $n=k+1$ (the case we do not consider in this paper). In this case, let $\Sym_{n-1}$ be the symmetric group on the set $\{2,3,4,\dots,n\}$, regarded as a subgroup of $\Sym_n$. Let $\Sym_n\times \Sym_{n-1}$ be defined as before. It can be shown that $\Aut(S_{n,n-1})\cong \Sym_{n}\times \Sym_{n-1}$. 
\end{remark}

\section{$k$-homogeneous and $k$-transitive groups}
Recall that a permutation group $G$ on the set $\Omega$ is {\em $k$-homogeneous} (resp. {\em $k$-transitive}) if it acts transitively on the set of $k$-combinations (resp. $k$-permutations) of $\Omega$, where the integer $k$ satisfies $1\le k\le n$ (denote $n:=|\Omega|$). 
Moreover, $G$ is {\em sharply $k$-transitive} if it acts regularly (i.e. transitively and freely) on the set of $k$-permutations; in this case, $|G|=P(n,k)$.

Fix $n$ and $k$ such that  $S_{n,k}\cong\Gamma(G,S)$ is a Cayley graph. For convenience of notation,  we regard the isomorphism $\varphi$  of Theorem \ref{thm:Aut is Sn} as an identity of $\Aut(S_{n,k})$ with the internal direct product:
$$\Sym_n\times \Sym_{k-1}=\Aut(S_{n,k})$$
(so we view $\Sym_n$ and $\Sym_{k-1}$ as subgroups of $\Aut(S_{n,k})$). By Theorem \ref{Sabidussi},  we can identify $G$ with a subgroup of $\Sym_n\times \Sym_{k-1}$.  We define $H$ to be the image $\pi_1(G)$ under the natural projection $\pi_1:\mathfrak{S}_n\times\mathfrak{S}_{k-1}\rightarrow \mathfrak{S}_n$, that is,
$$H=\{\mu\in \mathfrak{S}_n\::\:(\mu,\nu)\in G \text{ for some }\nu\in \mathfrak{S}_{k-1}\}.$$
Define $T=G\cap {\rm ker}(\pi_1)$. Since $T$ can be viewed as a subgroup of $\Sym_{k-1}$, we have that $|T|$ divides $(k-1)!$. By the first isomorphism theorem, we have $H\cong G/T$, thus $|H|=|G|/|T|$.

\begin{lemma}\label{H is k-homogeneous}
$H$ is $k$-homogeneous.
\end{lemma}
\begin{proof}
It suffices to prove that for any $k$-combination $\{a_1,\dots,a_k\}\subseteq \{1,\dots,n\}$, there is a permutation $\mu\in H$ such that $\{\mu(1),\dots,\mu(k)\}=\{a_1,\dots,a_k\}$. 

Since $G$ acts regularly on $S_{n,k}$, there exists $(\mu,\nu)\in\Sym_n\times \Sym_{k-1}$ that sends the vertex $\overline{\mathbf{e}}=[1,\dots,k]$ to the vertex $[a_1,\dots,a_k]$. That is,
$$[\mu(\nu^{-1}(1)),\dots,\mu(\nu^{-1}(k))]=[a_1,\dots,a_k].$$
 On the other hand, $\nu$ permutes $2,\dots, k$ and fixes $1$, so $\{\nu^{-1}(1),\dots,\nu^{-1}(k)\}=\{1,\dots,k\}$, thus
$\{\mu(1),\dots,\mu(k)\}=\{a_1,\dots,a_k\}.$
\end{proof}

In the rest of paper, $p$ always denotes a prime number, and $q$ denotes a prime power (that is, $q=p^m$ for some prime $p$ and positive integer $m$).

\begin{lemma}\cite{LW}\label{lemma:LW}
Let $G$ be $k$-homogeneous on a set of $n$ points, where $2\le k\le n/2$. Then 

(i) $G$ is $(k-1)$-homogeneous.

(ii) $G$ is $(k-1)$-transitive.

(iii) if $k\ge 5$, then $G$ is $k$-transitive.
\end{lemma}

\begin{lemma}\cite{K}\label{homogeneous not transitive}
Let $G$ be $k$-homogeneous but not $k$-transitive on a set of $n$ points, where $k\le n/2$. Then, up to permutation isomorphism, one of the following holds:

(i) $k=2$ and $G\le A\Gamma L(1,q)$ with $n=q\equiv 3\pmod 4$;

(ii) $k=3$ and $PSL(2,q)\le G\le P\Gamma L(2,q)$, where $n-1=q\equiv 3\pmod 4$;

(iii) $k=3$ and $G=AGL(1,8)$, $A\Gamma L(1,8)$ or $A\Gamma L(1,32)$; 

(iv) $k=4$ and $G=PSL(2,8)$, $P\Gamma L(2,8)$, or $P\Gamma L(2,32)$.
\end{lemma}

\begin{lemma}[Jordan, Zassenhaus,  {\cite[\S7.6]{DM}}] \label{lemma: sharply transitive}
A  finite sharply 2-transitive group is obtained from a finite near field $F$ and has order $|F|\times |F^\#|=q(q-1)$, with degree $q=|F|$.

A sharply 3-transitive group is either $PGL(2,q)$ (with order $(q+1)q(q-1)$, degree $q+1$, where $q$ is the order of the finite field $F$), or a twisted version of it (with the same order and degree).

The only sharply 4-transitive group is $M_{11}$ (with degree $11$).

The only sharply 5-transitive group is $M_{12}$ (with degree $12$).
\end{lemma}

\begin{lemma}[{\cite[\S3]{LPS}, \cite{DM}}]\label{lemma:k-transitive}
There are no 6-transitive groups other than $A_n$ and $S_n$.

The only 5-transitive groups are $A_n$ for all $n\ge 7$, $S_n$ for all  $n\ge 5$, $M_{12}\;(n=12)$ and $M_{24}\;(n=24)$.

The only 4-transitive but not 5-transitive groups are  $A_6 \;(n=6)$, $S_4\; (n=4)$, $M_{11}\; (n=11)$ and $M_{23}\; (n=23)$.

The 3-transitive but not 4-transitive groups are:
\begin{itemize}
\setlength\itemsep{0em}
\item $A_5\; (n=5)$, $S_3\;(n=3)$,

\item $AGL(d,2)\; (n=2^d)$, 

\item $2^4.A_7\; (n=2^4)$, 

\item $M_{11}\; (n=12)$, 

\item $M_{22}$ or $M_{22}.2 (=\Aut(M_{22}))$  $(n=22)$, 

\item or a 3-transitive subgroup of $P\Gamma L(2,q)\;\, (n=q+1)$. 
\end{itemize}
\end{lemma}

\section{Proof of Theorem \ref{main theorem}}
Recall that $S_{n,n-2}$ is Cayley for $n\ge 3$ (see \cite[Proposition 4]{CLLSS}). So from now on we assume
 $k\ge 2$ and $n\ge k+3$.

\begin{lemma}\label{sharp}
(i) If $H$ is $k$-transitive, then $H$ is sharply $k$-transitive.

(ii) If there exists a sharply $k$-transitive group $H\le\Sym_n$, then $S_{n,k}$ is Cayley.
\end{lemma}
\begin{proof}
(i) Since $H$ acts transitively on the set of $k$-permutations of $\Omega$ and $|H|\le |G|=P(n,k)$ (which is the number of $k$-permutations), this $H$-action must be regular. So $H$ is sharply $k$-transitive.

(ii) Let $G=\{(\mu,1)|\mu\in H\}\le \Sym_n\times\Sym_{k-1}$. Then $G$, which is isomorphic to $H$, acts regularly on $V(S_{n,k})$, thus by Theorem \ref{Sabidussi},  $S_{n,k}$ is Cayley.
\end{proof}

\subsection{The case $k=2$}
We need to show that  $S_{n,2}$ is Cayley if and only if $n$ is a prime power. (It is proved in \cite{CLLSS}, but we give a different proof here.)

For the ``if'' part: assume $n$ is a prime power. By Lemma \ref{lemma: sharply transitive} there exists a sharply 2-transitive group  $G\le \Sym_n$. Then the conclusion follows from Lemma \ref{sharp}.

For the ``only if'' part: $G\le \Sym_n\times\Sym_1\cong\Sym_n$, so $H\cong G$ is sharply 2-transitive. Then the conclusion follows from Lemma \ref{lemma: sharply transitive}.

\subsection{The case $k=3$}
We need to show that  $S_{n,3}$ is Cayley if and only if $n-1$ is a prime power. 

The ``if'' part  follows from Lemma \ref{lemma: sharply transitive} and \ref{sharp} using a similar argument as in the above case $k=2$.
(It is also proved in \cite{CLLSS}.)

For the ``only if'' part:  by Lemma \ref{H is k-homogeneous}, $H$ is $3$-homogeneous. Since $n\ge k+3=2k$, by Lemma \ref{homogeneous not transitive}, it suffices to discuss the following three cases:

(i) $H$ is $3$-transitive. By Lemma \ref{sharp}, $H$ is sharply $3$-transitive. Lemma \ref{lemma: sharply transitive} then implies that $n-1$ is a prime power.

(ii) $PSL(2,q)\le H\le P\Gamma L(2,q)$ and $n-1=q$ is  a prime power.

(iii) $H=AGL(1,8)$, $A\Gamma L(1,8)$ or $A\Gamma L(1,32)$. Then $n=8$ or $32$, so $n-1=7$ or $31$, which is prime in either case.

\subsection{The case $4\le k\le n/2$} We need to show that  $S_{n,k}$ is Cayley if and only if $(n,k)=(9,4), (11,4), (12,5), (33,4)$. For the ``if'' part, the cases $(n,k)=(11,4)$ and $(12,5)$ are clear because there exist sharply $k$-transitive groups of degree $n$ (Lemma \ref{lemma: sharply transitive}); the cases $(9,4), (33,4)$ will be proved later.

For the ``only if'' part: by Lemma \ref{H is k-homogeneous}, $H$ is $k$-homogeneous. By Lemma \ref{homogeneous not transitive}, it suffices to discuss the following two cases:

(i) $H$ is $k$-transitive. By Lemma \ref{sharp}, $H$ is sharply $k$-transitive. Lemma \ref{lemma: sharply transitive} then implies that $(n,k)=(11,4)$ or $(12,5)$.

(ii) $k=4$ and $H=PSL(2,8)$, $P\Gamma L(2,8)$, or $P\Gamma L(2,32)$. The corresponding $(n,k)=(9,4)$ or $(33,4)$.

\subsection{The case $n/2< k \le n-3$} We need to show that  $S_{n,k}$ is Cayley if and only if $(n,k)=(9,6)$, $(33,30)$, and possibly some $(2^d,2^d-3)$ for large $d$. We will prove the ``if'' part later.

For the ``only if'' part: we consider two cases:\\

(a) $H$ is $3$-transitive. First, recall that except the alternating group $A_n$, any proper subgroup of $\Sym_n$ has index at least $n$. So it is ``difficult'' to have a subgroup of $\Sym_n$ with a small index. Along the line is the following much stronger result which will play an important role in this paper.

\begin{theorem}\cite[Theorem 5.2B]{DM}\label{5.2B}
Let $a\ge 5$, $1\le r\le a/2$, and suppose $T\le \Sym_a$ has index $|\Sym_a:T|<\binom{a}{r}$. Then one of the following holds:

(i) for some $\Delta\subseteq\{1,\dots,a\}$, with $|\Delta|<r$, we have $A_{(\Delta)}\le T\le S_{\{\Delta\}}$, where the group $A_{(\Delta)}$ consists of all elements in the alternating group $A_a$ that fix every element in $\Delta$, the group $S_{\{\Delta\}}$ consists of all elements in the symmetric group $\Sym_a$ that fix the set $\Delta$ (but may permute elements in $\Delta$). 

(ii) $a=2m$ is even, $T$ is imprimitive with two blocks of size $m$, and $|\Sym_a:T|=\frac{1}{2}\binom{a}{m}$.

(iii) one of the six exceptional cases happens (where $r$ is the minimum number satisfying the hypotheses of the theorem
\footnote{Even though this minimum condition on $r$ is not stated in \cite[Theorem 5.2B]{DM}, a Remark after Theorem 5.2B mentions that ``the group are listed only with the minimum $r$ for which they satisfy the hypotheses of the theorem''. Without the minimum condition, there could be more exceptional cases such as $(6,3,6)$, $(6,3,12)$, $(8,4,30)$.}
): $(a,r,|\Sym_a:T|)=(6,3,15), (5,2,6), (6,2,6), (6,2,12), (7,3,30), (8,3,30)$. 
\end{theorem}
\medskip

Now we discuss case-by-case the list given in Lemma \ref{lemma:k-transitive}:\\

If $H=A_n$ or $S_n$, then $|H|$ does not divide $P(n,k)$ since $k\le n-3$. So this case is impossible.\\

If $H=M_{11}$ ($n=11$), then $6\le k\le 8$, and $11\cdot 10\cdot9\cdot8=|H|=P(11,k)/t$, where $t|(k-1)!$. The only possible solution is $k=8$ and $t=7\cdot 6\cdot 5\cdot 4$. 
Assume $T\le \Sym_7$ has order $t=7\cdot 6\cdot5\cdot4$, thus $|\Sym_7:T|=6<7$, which is impossible. 
\\

If $H=M_{12}$ ($n=12$), then $6\le k\le 9$, and $12\cdot 11\cdot 10\cdot9\cdot8=|H|=P(12,k)/t$, where $t|(k-1)!$.  The possible solutions are $k=8$ and $t=7\cdot 6\cdot 5$, or $k=9$ and $t=7\cdot 6\cdot 5\cdot 4$. 

First, assume $T\le \Sym_7$ has order $t=7\cdot6\cdot5$, thus $|\Sym_7:T|=24$. Using notation in Theorem \ref{5.2B}, $(a,r,|\Sym_a:T|)=(7,3, 24)$, clearly not the case (ii) or (iii) in Theorem \ref{5.2B}. So it has to be case (i): there is $\Delta$ with $|\Delta|<3$ and $A_{(\Delta)}\le T$. But since $A_{(\Delta)}\cong A_{a-|\Delta||}$,  $|A_{(\Delta)}|=(a-|\Delta|)!/2$ is a multiple of $(a-2)!/2=5!/2=60$, therefore $A_{(\Delta)}\le T$ implies 
$60|\; 7\cdot 6\cdot 5$, a contradiction. 

Next,  assume $T\le \Sym_8$ has order $t=7\cdot6\cdot5\cdot 4$, thus $|\Sym_8:T|=48$. Using notation in Theorem \ref{5.2B}, $(a,r,|\Sym_a:T|)=(8,3, 48)$, clearly not the case (ii) or (iii) in Theorem \ref{5.2B}. So again it has to be case (i). As above, $|A_{(\Delta)}|$ is a multiple of $(a-2)!/2=6!/2=360$, and divides $7\cdot6\cdot5\cdot4$, a contradiction. So this case is impossible.\\

If $H=M_{23}$ ($n=23$), then $3\cdot 16\cdot 20\cdot21\cdot22\cdot23=|H|=P(23,k)/t$ where $t|(k-1)!$. The only possible solution is $k=20$, and $t=19\cdot 18\cdot\cdots \cdot 4/(3\cdot 16)$. 
Assume $T\le \Sym_{19}$ has order $t=19\cdot 18\cdot\cdots \cdot 4/(3\cdot 16)$,
 thus $|\Sym_{19}:T|=3\cdot 16\cdot 6$. Using notation in Theorem \ref{5.2B}, $(a,r,|\Sym_a:T|)=(19,3, 3\cdot16\cdot6)$, clearly not the case (ii) or (iii) in Theorem \ref{5.2B}. So again it has to be case (i). As above, $|A_{(\Delta)}|$ is a multiple of $17!/2$, and divides $19\cdot 18\cdot\cdots \cdot 4/(3\cdot 16)$, therefore $3\cdot 16\cdot 3|19\cdot 18$, a contradiction. So this case is impossible.\\

If $H=M_{24}$ ($n=24$), then $3\cdot 16\cdot 20\cdot21\cdot22\cdot23\cdot24=|H|=P(24,k)/t$ where $t|(k-1)!$. The possible solutions are 
$k=20$ and $t=19\cdot 18\cdot\cdots \cdot 5/(3\cdot 16)$,
or
$k=21$ and $t=19\cdot 18\cdot\cdots \cdot 4/(3\cdot 16)$.

First, assume $T\le \Sym_{19}$ has order  $t=19\cdot 18\cdot\cdots \cdot 5/(3\cdot 16)$, thus $|\Sym_{19}:T|=3\cdot 16\cdot24$. Using notation in Theorem \ref{5.2B}, $(a,r,|\Sym_a:T|)=(19,4, 3\cdot16\cdot24)$, clearly not the case (ii) or (iii) in Theorem \ref{5.2B}. So it has to be case (i): there is $\Delta$ with $|\Delta|<4$ and $A_{(\Delta)}\le T$. But since $A_{(\Delta)}\cong A_{a-|\Delta||}$,  $|A_{(\Delta)}|=(a-|\Delta|)!/2$ is a multiple of $(a-3)!/2=16!/2$, therefore $A_{(\Delta)}\le T$ implies 
$16!/2\;|\; 19\cdot 18\cdot\cdots \cdot 5/(3\cdot 16)$, that is, $3\cdot 16\cdot 12 | 19\cdot 18\cdot 17$, a contradiction. 

Next,  assume $T\le \Sym_{20}$ has order $t=19\cdot 18\cdot\cdots \cdot 4/(3\cdot 16)$, thus $|\Sym_{20}:T|=3\cdot16\cdot120$. Using notation in Theorem \ref{5.2B}, $(a,r,|\Sym_a:T|)=(20,5, 3\cdot16\cdot120)$, clearly not the case (ii) or (iii) in Theorem \ref{5.2B}. So again it has to be case (i). As above, $|A_{(\Delta)}|$ is a multiple of $(a-4)!/2=16!/2$, therefore $16!/2\;|\; 19\cdot 18\cdot\cdots \cdot 4/(3\cdot 16)$, that is,  $3\cdot 16\cdot 3 | 19\cdot 18\cdot 17$, a contradiction. 

So this case is impossible.\\

If $H=2^4.A_7$ ($n=2^4$), by Lemma \ref{lemma:k-transitive}, $H$ is not 4-transitive, so  $k=13$. But $t=P(16,13)/|H|=P(16,13)/8!$, which does not divide $12!$. So this case is impossible.\\

If $H=M_{11}$ ($n=12$), then $k=9$ (because $H$ is not $4$-transitive). Then $t=P(12,9)/|H|=P(12,9)/(11\cdot 10\cdot 9\cdot 8)=12\cdot 7\cdot 6\cdot 5\cdot 4$.
Assume $T\le \Sym_{8}$ has order  $t$, then $|\Sym_{8}:T|=4<8$, which is impossible.\\

If $H=M_{22}$ or $M_{22}.2$ ($n=22$), then $k=19$ (because $H$ is not $4$-transitive). Then $t=P(22,19)/|H|$, which is either $P(22,19)/(3\cdot 6\cdot22\cdot 21\cdot 20)$ or half of it. In either case, $19|t$, thus $t\nmid18!$. So this case is impossible.\\

If $H$ is a $3$-transitive subgroup of $P\Gamma L(2,q)$ ($n=q+1$), then $k=q-2$ (because $H$ is not $4$-transitive). Then $t=P(q+1,q-2)/|H|$ is a multiple of $P(q+1,q-2)/(rq(q^2-1))=(q-2)(q-3)\cdots4/r$ (recall that $q=p^r$). It follows from $t|(q-3)!$ that $q-2(=p^r-2)$ divides $6r$. From the inequality $6r\ge p^r-2\ge 2^r-2$ we get $1\le r\le 5$.  It is then easy to verify by hand that there are two solutions: $q=2^3$, $2^5$. Correspondingly, $(n,k)=(9,6)$ or $(33,30)$, as we expected.\\

If $H=AGL(d,2)$, the proof is lengthy and tricky, so we leave it to Section \ref{H=AGL(d,2)}.

\bigskip

(b) $H$ is not 3-transitive. By Lemma \ref{H is k-homogeneous}, $H$ is $k$-homogeneous, thus is also $(n-k)$-homogeneous, where $n-k\ge 3$. Then by Lemma \ref{lemma:LW}, $n=k+3$ and $H$ is $3$-homogeneous. Then $H$ is listed as (ii) or (iii) of Lemma \ref{homogeneous not transitive}. The case (ii) is discussed in (a). So we only need to discuss (iii).

If $H=AGL(1,8)$, then $n=8$, $k=5$, and
$$\dfrac{P(8,5)}{t}=|H|=56$$
where $t|4!$. But then $t=P(8,5)/56=120>4!$, and we have reached a contradiciton.\\

If $H=A\Gamma L(1,8)$, which has order 168, $n=8$, $k=5$, and
$$ \dfrac{P(8,5)}{t}=|H|=168$$
where $t| 4!$. But then $t=P(8,5)/168=40 >4!$, and we have reached another contradiction. \\

If $H=A \Gamma L(1,32)$, $n=32$, $k=29$, and
$$\dfrac{P(32,29)}{t}=|H|=32\cdot 31\cdot 5$$
where $t| 28!$. But then $t=P(32,29)/(32\cdot 31\cdot 5)>28!$, and we have reached another contradiction.

\subsection{Constructive proof of Cayleyness} In this subsection, we shall prove that $S_{9,4}$, $S_{9,6}$, $S_{33,4}$, $S_{33,30}$ are Cayley.

First we define $\lambda$-transitivity introduced by W. Martin and B. Sagan	\cite{MS}. Let $\Omega$ be a set with $n$ elements and let $\lambda=(\lambda_1,\dots,\lambda_m)$ be a partition. A permutation group $G$ on the set $\Omega$ is called {\em $\lambda$-transitive} if $G$ acts transitively on the set of ordered tuples $(P_1,\dots,P_m)$  of pairwise disjoint subsets of $\Omega$ satisfying $|P_i|=\lambda_i$ for $1\le i\le m$ (we say that the ordered tuples have type $\lambda$). Of course, $\lambda$-transitivity does not change if we permute numbers in $\lambda$. 

For convenience, we say that $G$ is {\em sharply $\lambda$-transitive} (resp. {\em $\lambda$-free}) if $G$ acts regularly (resp. freely) on the above set of ordered tuples. Since the number of such ordered tuples is $n!/\prod \lambda_i!$, it is obvious that if $|G|=n!/\prod \lambda_i!$, then the following three conditions are equivalent:
\begin{itemize}
\item $G$ is sharply $\lambda$-transitive; 
\item $G$ is $\lambda$-transitive;
\item $G$ is $\lambda$-free.
\end{itemize}

The following lemma follows from the recent classification of $\lambda$-transitive groups (see \cite{DoM,AAC}); nevertheless, we include a straightforward proof here for self-containedness.
\begin{lemma}\label{lambda transitive}
The following permutation groups are sharply $\lambda$-transitive: $PSL(2,8)\le\Sym_9$ for $\lambda=(5,3,1)$, and $P\Gamma L(2,32)\le\Sym_{33}$ for $\lambda=(29,3,1)$.
\end{lemma}
\begin{proof}
Since $|PSL(2,8)|=9\cdot 8\cdot 7=9!/(5!3!1!)$ and $|P\Gamma L(2,32)|=5\cdot 33\cdot 32\cdot 31=33!/(29!3!1!)$, it suffices to prove these two groups are $\lambda$-free. 

For $PSL(2,8)$:  denote the finite field $\mathbb{F}_8=F_2[z]/(z^3+z+1)$,  
$$\bar{0}=\begin{bmatrix}
0\\1
\end{bmatrix}
,
\bar{1}=\begin{bmatrix}
1\\1
\end{bmatrix}
,
\bar{z}=
\begin{bmatrix}
z\\ 1
\end{bmatrix}
,
\infty=\begin{bmatrix}
1\\0
\end{bmatrix}
\in \mathbb{P}^1(\mathbb{F}_8).$$ 
Recall that $\mu\in PSL(2,8)$ acts on $\mathbb{P}^1(\mathbb{F}_8)$ by matrix multiplication
$\begin{bmatrix}
x\\y
\end{bmatrix}
\mapsto
\begin{bmatrix}
\alpha&\beta\\ \gamma&\delta
\end{bmatrix}
\begin{bmatrix}
x\\y
\end{bmatrix}
=
\begin{bmatrix}
\alpha x+\beta y\\ \gamma x+\delta y
\end{bmatrix}
$. 
Assume that $\mu$ fixes $\infty$ and permutes $\bar{0}, \bar{1}, \bar{z}$ (thus permutes the remaining 5 elements). We want to show that $\mu={\rm id}$. 
Since $\mu(\infty)=\infty$, we have $\gamma=0$, and without loss of generality we assume $\delta=1$. Then 
$$\{\beta, \alpha+\beta,\alpha z+\beta\} = \{0,1,z\}.$$
A case-by-case computation shows that $\alpha=1, \beta=0$, thus $\mu={\rm id}$.
\medskip

For $P\Gamma L(2,32)$: denote the finite field $F_{32}=F_2[z]/(z^5+z^2+1)$ and $\bar{0}, \bar{1}, \bar{z}$ as above. Recall that 
 $\mu$ acts by
$\begin{bmatrix}
x\\y
\end{bmatrix}
\mapsto
\begin{bmatrix}
\alpha&\beta\\ \gamma&\delta
\end{bmatrix}
\begin{bmatrix}
x^\sigma\\y^\sigma
\end{bmatrix}
=
\begin{bmatrix}
\alpha x^\sigma+\beta y^\sigma\\ \gamma x^\sigma+\delta y^\sigma
\end{bmatrix}
$ for some $\sigma\in \Aut(\mathbb{F}_{32})$, where $\sigma=\varphi^m$ for some $0\leq m\leq 4$ where $\varphi$ is the Frobenius automorphism defined by $x\mapsto x^2$. In particular $\sigma(z)\in \{z, z^2, z^4, z^8, z^{16}\}=\{z, z^2, z^4, z^3+z^2+1, z^4+z^3+z+1\}$.   
 
As above, assume $\mu$ fixes $\infty$ and permutes $\bar{0}, \bar{1}, \bar{z}$; and we can assume that $\gamma=0, \delta=1$. Then
$$\{\beta, \alpha+\beta,\alpha z^\sigma+\beta\} = \{0,1,z\}.$$
We do a case-by-case computation as follows:

If $(\beta, \alpha+\beta,\alpha z^\sigma+\beta)=(0,1,z)$: then $\alpha=1$, $z^\sigma=z$, $\sigma={\rm id}$,  thus $\mu={\rm id}$.

If $(\beta, \alpha+\beta,\alpha z^\sigma+\beta)=(0,z,1)$: then $\alpha=z$, $zz^\sigma=1$, $z^\sigma=z^{-1}=z^4+z$, which is impossible. 

If $(\beta, \alpha+\beta,\alpha z^\sigma+\beta)=(1,0,z)$: then $\alpha=\beta=1$, $z^\sigma+1=z$, $z^\sigma=z+1$, which is impossible. 

If $(\beta, \alpha+\beta,\alpha z^\sigma+\beta)=(1,z,0)$: then $\beta=1$, $\alpha=z+1$, $zz^\sigma+1=0$, $z^\sigma=z^4+z$, which is impossible. 

If $(\beta, \alpha+\beta,\alpha z^\sigma+\beta)=(z,0,1)$: then $
\alpha=\beta=z$, $zz^\sigma+z=1$, $z^\sigma=z^{-1}+1=z^4+z+1$, which is impossible. 

If $(\beta, \alpha+\beta,\alpha z^\sigma+\beta)=(z,1,0)$: then $\beta=z$, $\alpha=z+1$, $(z+1)z^\sigma+z=0$, which is impossible. 
\end{proof}

\begin{proposition}
Given $k\ge 2$, $n\ge k+2$, define $\lambda=(n-k,k-1,1)$. If there exists a permutation group $H\le \Sym_n$ that is sharply $\lambda$-transitive, then $S_{n,k}$ is Cayley.

As a consequence, $S_{9,4}$, $S_{9,6}$, $S_{33,4}$, $S_{33,30}$ are Cayley. 
\end{proposition}
\begin{proof}
Define $G=H\times \Sym_{k-1}\le \Sym_n\times \Sym_{k-1}$. 
Since $|G|=|H|\cdot (k-1)!=(n)!/(n-k)!=V(S_{n,k})$, 
 it suffices to show that $G_{\bar{\mathbf{e}}}$ is trivial. Then $S_{n,k}$ is Cayley by Corollary \ref{Sabidussi corollary}. 

 Assume $(\mu,\nu)\in G_{\bar{\mathbf{e}}}$, that is, $\overline{\mu\nu^{-1}}=\bar{\mathbf{e}}$. Then $(\mu(2),\dots,\mu(k))=(\nu(2),\dots,\nu(k))$ is a permutation of $\{2,\dots,k\}$, and $\mu(1)=\nu(1)=1$, so $\mu$ fixes the ordered tuple $(\{2,\dots,k\},\{1\},\{k+1,\dots,n\})$, which has type $(n-k,k-1,1)$. Then $\mu={\rm id}$ (thus $\nu={\rm id}$), since $H$ is sharply $\lambda$-transitive. 

The consequence follows from Lemma \ref{lambda transitive}. Indeed, for $S_{9,4}$, we have $\lambda=(5,3,1)$ and Lemma \ref{lambda transitive} asserts that $H=PSL(2,8)$ is sharply $(5,3,1)$-transitive, thus $S_{9,4}$ is Cayley. For $S_{9,6}$, we have $\lambda=(3,5,1)$; since $\lambda$-transitivity does not depend on the order of numbers listed in $\lambda$,  $H=PSL(2,8)$ is also sharply $(3,5,1)$-transitive, therefore $S_{9,6}$ is also Cayley. The Cayleyness of $S_{33,4}$ and $S_{33,30}$ can be proved similarly. 
\end{proof}

\section{The case $H=AGL(d,2)$} \label{H=AGL(d,2)}

In this section, we deal with the case $H=AGL(d,2)$.
In this case,
$n=2^d$ and $k=2^d-3$ (because $H$ is not $4$-transitive), where $d\ge 3$, and 
\begin{equation*}\label{agl}\dfrac{P(2^d,2^d-3)}{t}=|H|=2^d(2^d-1)(2^d-2^1)\cdots(2^d-2^{d-1}) \end{equation*}
Assume a subgroup $T\le S_{2^d-4}$ has order
\begin{equation}\label{t}
|T|=t=\frac{P(2^d,2^d-3)}{2^d(2^d-1)(2^d-2^1)\cdots(2^d-2^{d-1})},
\end{equation}
thus 
\begin{equation}\label{divide(2^d-4)!}
t|(2^d-4)!
\end{equation}
 We want to get some contradiction to conclude that this case is impossible. Below we give a short ``proof'' that relies on a conjecture, and a complete but longer proof.

\subsection{A conjectural proof}
Here we present a simple ``proof'' that relies on a conjecture in number theory. Indeed, if $H$ exists, then from the above we have
\begin{equation*}\label{agl}\dfrac{2^d(2^d-1)\cdots (2^d-(2^d-4))}{t}=|H|=2^d(2^d-1)(2^d-2^1)\cdots(2^d-2^{d-1}) \end{equation*}
so
 $$2^d! = 6t\cdot 2^d(2^d-1)\cdots (2^d-2^{d-1})\Big| 6(2^d-4)!  \cdot 2^d(2^d-1)(2^d-2)\cdots (2^d-2^{d-1}) $$
or,
 $$ 2^d-3 \Big| 6(2^d-2^2)\cdots(2^d-2^{d-1})=6\cdot 2^{2+3+\cdots+d-1}(2^{d-2}-1)(2^{d-3}-1)\cdots(2^2-1).$$
 Since $\gcd(2^d-3,3)=\gcd(2^d-3,2)=\gcd(2^d-3,2^{d-2}-1)=1$, the nonexistence of $H$ is equivalent to
\begin{equation}\label{eq:divisibility false}
2^d-3\nmid (2^{d-3}-1)(2^{d-4}-1)\cdots(2^3-1).
\end{equation}
Moreover, the nonexistence of such $H$ also follows from a conjecture in number theory. To give the motivation, let us recall Bang's Theorem (a corollary of Zsigmondy's Theorem) which asserts:

``For any positive integer $d\neq 1, 6$, there is a prime factor of $2^d-1$ which is not a factor of $2^i-1$ for any $i<d$.''

For a sequence of nonzero integers $(a_n)_{n\ge 1}$, the \emph{Zsigmondy set} of the sequence is the set
$\{n\geq 1 | \textrm{ every prime factor of $a_{n}$ divides $a_m$ for some some $m<n$} \}$. For many interesting sequences, the Zsigmondy sets are finite (and very small). For example, the Zsigmondy set of $(2^d-1)_{d\ge1}$ is $\{1,6\}$. 

\begin{conjecture}
For any $d\ge 8$, there is a prime divisor of $2^d-3$ that does not divide $2^i-3$ for any $i<d$. (After checking the simple cases $d<8$, this conjecture is equivalent to the statement that the Zsigmondy set of the sequence $(2^i-3)$ is $\{1,2,7\}$.)
\end{conjecture}
\begin{proof}[Proof that the conjecture implies \eqref{eq:divisibility false}]
Let $p$ be a prime divisor of $2^d-3$  that does not divide $2^i-3$ for any $i<d$. To show \eqref{eq:divisibility false}, it suffices to show that $p\nmid 2^j-1$ for $3\le j\le d-3$. Assume the contrary that $p\nmid 2^j-1$ for some $3\le j\le d-3$. Then $p$ divides $(2^d-3)-2^{d-j}(2^j-1)=2^{d-j}-3$, a contradiction. 
\end{proof}

\begin{remark}(We thank P.~Ingram, J.~Silverman, and T.~Tucker for the following comments.) Even though the conjecture seems similar to Bang's Theorem stated above, there is an essential difference. Namely, $1$ is the identity in the multiplicative group, and a prime factor of $2^n-1$ automatically divides $2^{nk}-1$ for any $k$; in contrast, a prime number $p$ dividing $2^n-3$ only tells us (a priori) that $p$ divides $2^{nk}-3^k$ for any $k$, which is not necessarily another term in the sequence. This problem has been considered by experts, but is still open, even assuming the abc conjecture (see \cite{GNT} for an interesting connection between the abc conjecture and primitive divisor problems). 

Originally we were only able to check the conjecture for $d\le 320$ using a computer. As suggested by a referee, we replaced the prime factorization method by the gcd method, and were able to check the conjecture for $d\le 20000$.
Therefore \eqref{eq:divisibility false} holds for $d\le 20000$. 
\end{remark}

\subsection{A proof using Theorem \ref{5.2B}} The idea is to find a reasonably small integer $r$ such that $|\Sym_{k-1}:T|<\binom{k-1}{r}$ and that $|A_{(\Delta)}|$ does not divide $|T|=t$, then draw the expected conclusion using Theorem \ref{5.2B}. 

 Let $$n=2^d, \quad k=2^d-3, \quad r=\frac{d^2-d}{2}-2.$$ 
Let $t$ be defined as in \eqref{t}.
It is straightforward to check that \eqref{divide(2^d-4)!} does not hold for $d<8$. So in the following we assume $d\ge 8$. An easy inductive argument shows that 
$$d^2\le 2^{d-2}, \textrm{ for all } d\ge 8.
$$
It then follows that
$d^2\leq(k+3)/4$, thus $r+1<d^2/2-1<k/8$.

\begin{lemma}\label{lemma1}
 $(k-1)!/t<\binom{k-1}{r}$.
\end{lemma}
\begin{proof}
By the definition of $t$ as in \eqref{t},
$$\aligned
(k-1)!/t
&=\frac{(2^d-4)!3!2^d(2^d-1)(2^d-2^1)\cdots(2^d-2^{d-1})}{(2^d)!}
=\frac{3!(2^d-2^2)(2^d-2^3)\cdots(2^d-2^{d-1})}{2^d-3}\\
&=\frac{6(2^d-2^{d-1})(2^d-2^2)}{2^d-3}
\cdot (2^d-2^3)\cdots(2^d-2^{d-2})\le 3(2^d)(2^d-2^3)\cdots(2^d-2^{d-2})\\
&<3(2^d)^{d-3}.
\endaligned$$
On the other hand,
$\binom{k-1}{r}\ge(\frac{k-1}{r})^r$.
So it suffices to show
$$3(2^d)^{d-3}\le (\frac{k-1}{r})^r.$$
Equivalently,
$$(3\cdot 2^{d^2-3d})^{1/r}\le \frac{k-1}{r}.$$
This inequality can be shown as follows: the left hand side satisfies
$$(3\cdot 2^{d^2-3d})^{1/r}\le (2^{d^2-3d+2})^{1/r}=2^{2(d^2-3d+2)/(d^2-d-4)}<4,$$
and the right hand side satisfies
$$\frac{k-1}{r}=\frac{2(2^d-4)}{d^2-d-4}>=4\frac{2^{d-1}-2}{d^2}>4\frac{2^{d-2}}{d^2}\ge4.$$
\end{proof}

\begin{lemma}\label{lemma2}
$(k-r)!/2$ does not divide $t$. 
\end{lemma}
\begin{proof}
For any integer $i$, denote by $v_2(i)$ the largest integer $v$ such that $2^v|i$. For a nonzero rational number $i/j$, define $v_2(i/j)=v_2(i)-v_2(j)$. Then it suffices to show that
$$v_2\big((k-r)!/(2t)\big)>0.$$
Indeed,
$$\aligned
v_2((k-r)!/(2t))&=v_2\Big(\frac{(k-r)!2^d(2^d-1)(2^d-2^1)\cdots(2^d-2^{d-1})}{2\cdot2^d(2^d-1)\cdots5\cdot4}\Big)\\
&=v_2\Big((2^d-1)(2^d-2^1)\cdots(2^d-2^{d-1})/2\Big)+v_2\Big(\frac{(2^d-3-r)!}{(2^d-1)\cdots5\cdot4}\Big)\\
&=0+1+2+\cdots+(d-1)-1-v_2\Big(\frac{(2^d-1)\cdots5\cdot4}{(2^d-3-r)!}\Big)\\
&=d(d-1)/2-v_2\Big((2^d-1)(2^d-2)\cdots(2^d-2-r)\Big)\\
&=(r+2)-v_2\Big((r+2)!\Big) \quad\quad \textrm{(because $v_2(2^d-i)=v_2(i)$ for $1\le i\le 2^d-1$)}\\
&=(r+2)-\sum_{i=1}^\infty\big\lfloor \frac{r+2}{2^i}\big\rfloor
>(r+2)-\sum_{i=1}^\infty\frac{r+2}{2^i}=(r+2)-(r+2)=0.
\endaligned
$$
Therefore, $(k-r)!/2$ does not divide $t$. 
\end{proof}

Finally, we can now prove that $H=AGL(d,2)$ is impossible. Thanks to Lemma \ref{lemma1}, we only need to consider the three cases (i)--(iii) in Theorem \ref{5.2B}. Since $a=k-1=2^d-3\ge 2^8-3>8$, we do not need to consider Case (iii). 

For Case (ii): first note that $r+1<k/8$ implies  
$\frac{r+1}{k-1-r}<1/2$. Then
$$|\Sym_a:T|=(k-1)!/t<\binom{k-1}{r}=\frac{r+1}{k-1-r}\binom{k-1}{r+1}<1/2\binom{k-1}{(k-1)/2}$$ 
so this case is also impossible.

For Case (i): there exists some $\Delta\subseteq\{1,\dots,k-1\}$ such that $|\Delta|<r$ and $A_{(\Delta)}\le T$. Since $A_{(\Delta)}\cong A_{k-1-|\Delta|}$, and $|A_{k-1-|\Delta|}|=(k-1-|\Delta|)!/2$, it follows that $(k-1-|\Delta|)!/2$ divides $t$, therefore $(k-r)!/2\;|\;t$, contradicting Lemma \ref{lemma2}.  So this case is also impossible. 

This completes the proof.
\bigskip

\medskip\noindent{\it Acknowledgments.}  We are grateful to the anonymous referees for carefully reading
through the manuscript and giving us many constructive suggestions to improve the presentation.
We would also like to thank P.~Ingram, J.~Silverman, and T.~Tucker for their helpful comments on primitive divisors and Zsigmondy sets. Magma Computational Algebra System \cite{Magma} did many computations that are essential to our project.

\end{document}